\documentclass{amsart}
\usepackage{amssymb,latexsym}
\theoremstyle{plain}
\newtheorem{theorem}{Theorem}

\newtheorem{lemma}{Lemma}
\theoremstyle{definition}

\newtheorem{example}{Example}

\newtheorem{remark}{Remark}

\begin{document}

\title[primitive linear series]
{On the existence of primitive pencils for smooth curves}
\author{E. Ballico}
\address{Dept. of Mathematics\\
 University of Trento\\
38123 Povo (TN), Italy}
\email{ballico@science.unitn.it}
\thanks{The author was partially supported by MIUR and GNSAGA of INdAM (Italy).}

\begin{abstract}
Let $C$ be a smooth curve with gonality $k\ge 6$ and genus $g\ge 2k^2+5k-6$. We prove that $W^1_d({C})$ has the expected dimension
and that the general element of any irreducible component of $W^1_d({C})$ is primitive if either $g-k+4\le d\le g-2$ or $d=g-k+3$
and either $k$ is odd or $C$ is not a double covering of a curve of gonality $k/2$ and genus $k-3$. Even in the latter case we prove the existence of a complete and primitive
$g^1_{g-k+3}$.
\end{abstract}

\maketitle

A line bundle $L$ on a smooth curve $C$ of genus $g\ge 4$ is said to be {\emph{primitive}} if it is spanned and both $L$ and $\omega _C\otimes L^\vee$ are spanned,
i.e. if it is spanned and $h^0(L(q)) =h^0(L)$ for all $q\in C$ (sometimes one also imposes that $L\ne \mathcal {O} _C$ and $L\ne \omega _C$) (\cite{ckm}, \cite{ckm1}). Since $L$ is primitive
if and only if $\omega _C\otimes L^\vee$ is primitive, to study primitive line bundles on $C$ it is sufficient to know the ones with $0 < \deg (L) \le g-1$.
Let $C$ be a smooth curves of genus $g\ge 4$ with gonality $k$. If either $g\ne k(k-1)/2$ or $C$ is not isomorphic to a smooth plane curve, then
$C$ has a complete and primitive $g^1_k$. For very low $k$ or for a general smooth curve of genus $g$ the Brill-Noether theory of $C$ is well-known and it gives a complete description
of the complete and primitive $g^r_d$ on $C$ (\cite{ckm}, \cite{ckm1}, \cite{ckm2}, \cite{cm}). If $g\ge 11$ and $k\ge 5$, then a general element of any irreducible component of $W^1_{g-2}({C})$ is primitive
(\cite[Proposition II.0]{v}). In this note we consider the existence of complete and primitive $g^1_d$ for all $d$ near $g-1$ and prove the following result.

\begin{theorem}\label{or1}
Fix an integer $k\ge 6$ and set $g(k):=2k^2+5k-6$. Fix any integer $g\ge g(k)$, any smooth curve $C$ with gonality $k$
and any integer $d$ with $g-k+3 \le d\le g-2$.

\quad (a) $C$ has a complete and primitive $g^1_d$ and every irreducible component of $W^1_d({C})$ has dimension $2d-g-2$.

\quad (b) Assume that either $k$ is odd or $d>g-k+3$ or $d=g-k+3$, $k$ is even, but $C$ is not a double covering of a smooth curve of genus $k-3$ and gonality $k/2$.
Then a general element of every irreducible component of $W^1_d({C})$ is primitive.

\quad ({c}) Assume that $k$ is even and that $C$ is a double covering of a smooth curve of genus $k-3$ and
gonality $k/2$.

\quad (c1) There exist an irreducible component of $W^1_{g-k+3}({C})$ whose general member
is base point free and an irreducible component of $W^1_{g-k+3}({C})$ whose general member has $g-2k+3$ base points. 

\quad (c2) Let $\Gamma$ be any irreducible component of $W^1_{g-k+3}({C})$; if the general element of $\Gamma$ is base point free, then it is primitive.\end{theorem}

\begin{remark}\label{or2}
Fix an integer $x\ge 3$. Set $w:= \lfloor x/2\rfloor$, $z:= \lfloor (x-1)/2\rfloor$, $g_1(x):= 2x(2x-z)-4x+z+2$ and $g_2(x):= 4x^2-2wx+w$. Fix an integer $g\ge \max \{g_1(x),g_2(x)\}$.
Let $C$ be a smooth curve of genus $g$ and gonality at least $x+3$. The interested reader may reformulate an analogous of Theorem \ref{or1} with $d =g-x$
and prove it following verbatim the proof of Theorem \ref{or1}. In the case $g-d=4$ of Theorem \ref{or1} it is sufficient to assume that $g\ge 64$ (see Theorem \ref{s1.0}).
\end{remark}

\begin{remark}\label{or3}
In the exceptional cases of Theorem \ref{or1} we have a description of the irreducible components of $W^1_{g-k+3}({C})$ whose general
element has base points. We have
$\dim (W^1_k({C}))=1$ and each element of these components is obtained from some $R\in W^1_k({C})$ adding a base locus of degree $g-2k+3$. 
Let $Y$ be any smooth curve of genus $g$ and gonality $k$ with $\dim (W^1_k({Y}))=1$. If $W^1_{g-k+3}(Y)$ has pure dimension
$g-2k+4$, then it has at least one component formed by pencils with a base locus of degree $g-2k+3$. Steps (a) and (b) of the proof of Theorems \ref{or1}
show that if $g\gg k$, then $k$ is even and $Y$ is a double covering of a smooth curve of genus $k-3$ and gonality $k/2$.
\end{remark}

Many thanks are due to E. Sernesi for stimulating and interesting conversations.

\section{The proofs}\label{Sp}

\begin{lemma}\label{s1}
Fix integers $g, x$ such that $x\ge 2$ and $g\ge 4x+3$. Let $C$ be a smooth curve of genus $g$. Let $T$ be an irreducible component of $W^1_{g-x}({C})$.

\quad (a) If $\dim (T) > g-2x-2$, then $\dim (W^1_{2x}({C})) \ge x-1$.

\quad (b) If a general element of $T$ has a base point, then $\dim (W^1_{2x+2}({C})) \ge x$.
\end{lemma}

\begin{proof}
Since $g \ge 2x+2$, Brill-Noether theory gives $W^1_{g-x}({C})\ne \emptyset$ and that each irreducible component of $W^1_{g-x}({C})$ has dimension at least $g-2x-2$ (\cite[Ch. IV]{acgh}).

First assume $\dim (T)\ge g-2x-1$. Set $d: =g-x$ and $j=x-1$. We have $g-2x-1 =d-2-j$, $d \ge j+2$ and $d \le g-1-j$ (the
latter inequality is an equality). By \cite{m} or \cite[Theorem 1]{h} we have $\dim (W^1_{2x}({C})) \ge x-1$. 

Now assume that a general element of $T$ has at least one base point. We get an irreducible component of $W^1_{g-x-1}({C})$ with dimension at least $g-2x-3$. Apply part (a) with the integer $x':= x+1$ instead of the integer $x$.
\end{proof}

\begin{proof}[Proof of Theorem \ref{or1}:]
Since $2d-g-2\ge 0$, Brill-Noether theory says that $W^1_d({C}) \ne \emptyset$ and that each irreducible component of $W^1_d({C})$ has dimension at least $2d-g-2$. Let $T$ be an irreducible component of $W^1_d({C})$ and let $R$ be a general element of $T$. As in the case of the general member of any irreducible component of any
$W^1_y({C})$ with $y\le g-1$ we have $h^0({R}) =2$. To prove Theorem \ref{or1} it is sufficient to prove that $\dim (T)=2d-g-2$, that $R$ is base point free and that $h^0(R({p}))=1$ for all $p\in C$. Let $f: X\to \mathbb {P}^1$ be any degree $k$ morphism.

\quad (a) Assume $\dim (T) > 2d-g-2$. The case $x=g-d$ of part (a) of Lemma \ref{s1} gives $\dim (W^1_{2g-2d}({C})) \ge g-d-1$. Let $\Gamma$ be any irreducible component of $W^1_{2g-2d}({C})$ with $\dim (\Gamma)
\ge g-d-1$. Let $R'$ be a general element of $\Gamma $. Since $R'$ is general in an irreducible component of some $W^1_y({C})$, $y\le g+1$, we have $h^0(R') =2$.
Let $s\ge 0$ be the degree of the base locus $B$
of $R'$. Varying $R'$ we get an irreducible family $\Gamma '\subseteq W^1_{2g-2d-s}({C})$ with $\dim (\Gamma ') \ge g-d-1$. Set $R'':= R'(-B)$. Let $u: C\to \mathbb {P}^1$ be the morphism associated
to $|R''|$ and let $\alpha : C\to \mathbb {P}^1\times \mathbb {P}^1$ be the morphism associated to $(f,u)$. If $\alpha$ is birational onto its image,
then $g\le k(2g-2d-s) -k-(2g-2d-s) +1 \le k(2k+6)-k-(2k+6)+1 <g(k)$, a contradiction. 
Hence $C$ is not birational onto its image, i.e. calling $C'$ the normalization of $\alpha ({C})$ we get a morphism $\beta ': C\to C'$ with $\beta : =\deg (\beta ')\ge 2$ and morphisms $f': C'\to \mathbb {P}^1$ and $u': C'\to \mathbb {P}^1$ such that $f = f'\circ \beta '$ and $u = u'\circ \beta '$. Since $f$ computes the gonality of $C$, we get that $C'$ has genus $q_{R'}>0$ and
that $C'$ has gonality $k/\beta$.

First assume $q_{R'} \ge 2$. Since $\Gamma$ is irreducible and $C$ has only finitely many non-constant morphisms to curves of genus between $2$ and $g-1$ by a theorem of de Franchis, we get that
$C'$, $\beta $ and $\beta '$ do not depend from the choice of $R'$. Since $h^0(R'')=2$, we have $h^0(C,L'')=2$, where $L''$ is the line bundle on $C'$ with $\beta ^{'\ast }(L'')\cong R''$.
Since $h^0(R'')=2$, we have $h^0(L'') =2$ and hence $(2g-2d-s)/\beta \le q+1$.
By \cite[Theorem 1]{fhl} we have $W^1_{(2g-2d-s)-\beta \lfloor (g-d-1)/2\rfloor}({C}) \ne \emptyset$
and hence $k\le (2g-2d-s) -\beta \lfloor (g-d-1)/2\rfloor \le (2g-2) -\beta \lfloor (g-d-1)/2\rfloor$. Since $\beta \ge 2$ and $g-d\le k-3$, we get $k \le 2g-2d -g+d+2
\le 2k-6-k+5$, a contradiction.

Now assume $q_{R'}=1$. Since $C'$ has gonality $k/\beta$, we get $\beta =k/2$. Hence $2g-2d-s$ is divisible by $k/2$. Since $d\le g-2$, we have $\dim (W^1_{2g-2d-s}({C})) \ge 3$ and hence $2g-2d-s >k$. Therefore $2g-2d-s \ge 3k/2$. Since $q_{R'}=1$, we have $h^0(C',L'') = \deg (L'') \ge 3 > h^0(R'')$, a contradiction.

\quad (b) In this step we prove that $R$ has no base points if one of the conditions in part (b) of the statement of Theorem \ref{or1} is satisfied. If $R$ has a base point, then part (b) of Lemma \ref{s1} with $x=g-d$
gives $\dim (W^1_{2g-2d+2}({C})) \ge g-d$.  Let $\Gamma _1$ be any irreducible component of $W^1_{2g-2d+2}({C})$ with $\dim (\Gamma _1)
\ge g-d$. Let $R'$ be a general element of $\Gamma _1$. Since $R'$ is general in an irreducible component of some $W^1_y({C})$, $y\le g+1$, we have $h^0(R') =2$.
Let $s\ge 0$ be the degree of the base locus $B$
of $R'$. Varying $R'':= R'(-B)$ we get an irreducible family $\Gamma  '\subseteq W^1_{2g-2d+2-s}({C})$ with $\dim (\Gamma ') \ge g-d$ and with $R''$
as its general member. Let $u: C\to \mathbb {P}^1$ be the morphism associated
to $|R''|$ and let $\alpha : C\to \mathbb {P}^1\times \mathbb {P}^1$ be the morphism associated to $(f,u)$. If $\alpha$ is birational onto its image,
then $g\le k(2g-2d+2-s) -k-(2g-2d+2-s) +1 \le k(2k+8)-k-(2k+8)+1 =g(k)-1$, a contradiction. 
Hence $C$ is not birational onto its image, i.e. calling $C'$ the normalization of $\alpha ({C})$ we get a morphism $\beta ': C\to C'$ with $\beta : =\deg (\beta ')\ge 2$ and morphisms $f': C'\to \mathbb {P}^1$ and $u': C'\to \mathbb {P}^1$ such that $f = f'\circ \beta '$ and $u = u'\circ \beta '$. Since $f$ computes the gonality of $C$, we get that $C'$ has genus $q_{R'}>0$ and
that $C'$ has gonality $k/\beta$. 

First assume $q_{R'} \ge 2$. Since $\Gamma$ is irreducible and $C$ has only finitely many non-constant morphisms to curves of genus between $2$ and $g-1$ by a theorem of de Franchis, we get that
$C'$, $\beta '$ and $\beta$ does not depend from the choice of $R'$. Since $h^0(R'')=2$, we have $h^0(C,L'')=2$, where $L''$ is the line bundle on $C'$ with $\beta ^{'\ast }(L'')\cong R''$.
Since $h^0(R'')=2$, we have $h^0(L'') =2$ and hence $(2g-2d+2-s)/\beta \le q+1$.
By \cite[Theorem 1]{fhl} we have $W^1_{2g-2d+2-2s-\deg (B_1)-\beta \lfloor (g-d)/2\rfloor}({C}) \ne \emptyset$
and hence 
$k\le (2g-2d+2-s) -\beta \lfloor (g-d)/2\rfloor$. Since $2 \le g-d \le k-3$, we get $k=g-d+3$, $\beta =2$, $s=0$, $g-d$ odd and $q\ge k-3$. We also get that
$C'$ has gonality $k/2$. Since $\dim (W^1_{k-2}(C')) \ge \dim (\Gamma _1) \ge k-3 \ge q$, we get $q=k-3$. We are in the exceptional case allowed in the
statement of Theorem \ref{or1}.

Now assume $q_{R'}=1$. Since $C'$ has gonality $k/\beta$, we get $\beta =k/2$. Hence $2g-2d-s$ is divisible by $k/2$. Since $d\le g-2$, we have $\dim (W^1_{2g-2d-s}({C})) \ge 3$ and hence $2g-2d-s >k$. Therefore $2g-2d-s=3k/2$. Since $q_{R'}=1$, we have $h^0(C',L'') = \deg (L'') = 3 > h^0(R'')$, a contradiction.

\quad ({c}) Assume the existence of $p_R\in C$ such that  $h^0(R(p_R)) =3$. Since $R$ is base point free and $h^0({R}) =2$, $M:= R(p_R)$ is base point free. Let $u: C\to \mathbb {P}^2$ be the morphism induced by $|M|$. Since $R$ is general in $T$, we get $\dim (W^2_{d+1}({C})) \ge 2d-g-3$ and $\dim (W^2_{d+1}({C})) \ge 2d-g-2$, unless the
same general $M$ comes from infinitely many pairs $(R_1,P_{R_1})$ with $R_1$ general in $T$.
 Since a smooth plane curve of degree $d+1$ has gonality $d>k$,
either $\deg (u) >1$ or $u({C})$ is a singular curve. First assume $\deg (u)=1$. Taking the a linear projection from one of the finitely many singular points of $u({C})$ we get $\dim (W^1_{d-1}({C})) \ge 2d-g-3$. The case $x_1:= x+1$ of part (a) of Lemma \ref{s1} gives $\dim (W^1_{2g-2d+2}({C})) \ge g-d$.
We are in the exceptional case described in step (b). Now assume $\deg (u)>1$. If  $\dim (W^2_{d+1}({C})) \ge 2d-g-2$, we get the same lower bound for $\dim (W^1_{d-1}({C}))$
taking a linear projection from any point of $u({C})$. Taking a linear projection we get a better estimate if either $\dim (W^2_{d+1}({C})) \le 2d-g-3$ or $\deg (u)\ge 3$
or $u({C})$ is singular. Now assume $\deg (u) =2$ and that $u({C})$ is smooth. Since $\deg (u({C})) =(d+1)/2 \ge (g-k+34)/2$,
we get that $u({C})$ has genus $q' \ge  (g-k+3)(g-k+2)/8$. Since $g \ge 2q'-1$ (Riemann-Hurwitz), we get a contradiction.

\quad (d) To conclude the proof we may assume that $k$ is even and the existence of a degree $2$ covering $\beta ':C\to C'$ with
$C'$ smooth of genus $k-3$ and gonality $k/2$.  By step (a) $W^1_{g-k+3}({C})$ has pure dimension $g-2k+4$. Remark \ref{or3} gives the existence of an irreducible component
of $W^1_{g-k+3}({C})$ whose general member has $g-2k+3$ base points. Brill-Noether theory gives  $\dim (W^1_t(C'))\ge 2t-k+1$ for all $t\in \mathbb {N}$
such that $k/2 \le t \le k-2$. Since $C'$ has gonality $k/2$, \cite[Theorem 1]{fhl} first gives
$\dim (W^1_{k/2}(C'))=1$ and then $\dim (W^1_t(C'))\le 2t-k+1$ for all $t\in \mathbb {N}$
such that $k/2 < t \le k-2$. Hence $C'$ has dimensionally the Brill-Noether theory for pencils of a general curve of genus $k-3$. This is enough
to carry over the proof of \cite[Theorem 0.1]{bk} (see the proofs of Lemmas 1.2, 1.3 and Theorem 0.1 in \cite{bk}).
Hence (with this observation concerning $C'$), \cite[Theorem 0.1]{bk} gives the existence of a degree $g-k+3$ morphism $f: C\to \mathbb {P}^1$ not composed with $\beta '$,
i.e.
such that the morphism $(\beta ',f): C\to \mathbb {P}^1\times \mathbb {P}^1$ is birational onto its image. We claim that we may take as $f$ a complete pencil.
This claim is true by \cite[Lemma 1.3]{bk}, which describes all the $W^1_{g-k+3}({C})$, $k-3 =p_a(C')$, whose general element has base points
and the existence of at least another components of $W^1_{g-k+3}({C})$ (\cite[first line of page 155]{bk}; it is the second line of page 155, which loses the completeness statement in
\cite[Theorem 0.1]{bk}). So we proved the existence of an irreducible component of $W^1_{g-k+3}({C})$ containing a base point free and complete $g^1_{g-k+3}$.

Let $\Gamma$ be any irreducible component of $W^1_{g-k+3}({C})$ containing a base point free and complete $g^1_{g-k+3}$, $\delta$.
Since $\delta$ is complete, the general element of $\Gamma$ is complete and base point free. Let $R$ be a general element
of $\Gamma$. Assume the existence of $p_R\in C$ such that  $h^0(R(p_R)) =3$. Since $R$ is base point free and $h^0({R}) =2$, $M:= R(p_R)$ is base point free.
Since $R$ is general in $\Gamma$, we get $\dim (W^2_{g-k+4}({C})) \ge g-2k+3$
and $\dim (W^2_{g-k+4}({C})) \ge g-2k+4$, unless a general $M$ comes from infinitely many pairs $(R,p_R)$ .
 Let $u: C\to \mathbb {P}^2$ be the morphism induced by $|M|$. First assume
$\deg (u) =1$. Since $g-k+4 >k+1$
and a smooth plane curve of degree $g-k+4$ has gonality $g-k+3>k$,
$u({C})$ is a singular curve. Therefore taking a linear projection from a singular point of $u({C})$ we obtain $\dim (W^1_{g-k+2}({C})) \ge g-2k+3$ (because
$u({C})$ has only finitely many singular points). Write $\dim (W^1_{g-k+2}({C})) = g-k-1-j+e$ with $e\ge 0$ and $j =k-4$.
By \cite[Theorem 1]{h} we have $\dim (W^1_{2k-6-2e}({C})) = k-1-e$ and (hence, even if $e>0$ by \cite[Theorem 1]{fhl}) we have $\dim (W^1_{2k-6}({C})) \ge k-4$. This
is the case handled in step (a). Now assume $\deg (u) >1$ and call $C''$ the normalization of $u({C})$, $q$ its genus, and $v: C\to C''$
the morphism through which $u$ factors. Write $M = v^\ast (L)$ with $L$ base point free line bundle on $C''$ with $h^0(C,L) =3$. If $\deg (u)\ge 3$, then fixing $o\in u({C})_{\mathrm{reg}}$ and taking the linear
projection from $o$
we get $\dim (W^1_{g-k+1}({C})) \ge g-2k+3$ (the same for all $M$, because $q\ge 2$ and we may apply a theorem of de Franchis). Write $\dim (W^1_{g-k+1}({C})) = g-k-1-j+e'$ with $e'\ge 0$ and $j=k-4$. We get $\dim (W^1_{2k-8}({C})) \ge k-3$
and conclude. Now assume $\deg (u) =2$. In this case $(g-k+4)/2\in \mathbb {Z}$. If $u({C})$ is singular, then a linear projection from one of its singular points
gives $\dim (W^1_{g-k}({C}))\ge g-2k+3$ and so  $\dim (W^1_{g-k+3}({C}))\ge g-2k+6$, contradicting step (a).
Hence $u({C})$ is smooth and so it has genus $q':= (g-k+2)(g-k)/8$. Riemann-Hurwitz gives $g\ge 2q'-1$, a contradiction.\end{proof}

\begin{example}\label{or4}
Fix an even integer $k\ge 6$ and set $x:= k-3$. Let $C'$ be any smooth curve of genus $x$ and gonality $k/2$. We have $\dim (W^1_{k/2}(C')) =1$.
Let $u: C\to C'$ be a degree $2$ covering of genus $g \ge 3x+4$. $W^1_{g-x}({C})$ contains the $(g-2x-2)$-dimensional family of $g^1_{g-x}$ obtained from
the pull-backs of the elements of $W^1_{k/2}(C')$ adding $g-2x-3$ base points. Since $g > 2x+k$, any base point free pencil on $C$ of degree $ < k$
is the pull-back of a base point free pencil on $C'$ by the Castelnuovo-Severi inequality (\cite{k}). Hence $C$ has gonality $k$. Hence the exceptional cases
in Theorem \ref{or1} arises. \end{example}

See also \cite{bkp}  (resp. \cite{bks}) for the existence of spanned pencils on curves which are double (resp. multiple) coverings. By \cite[Theorem 2.2]{bk2} for every integer
$d\ge g-k+2$ every $k$-gonal
curves of genus $g > (3k - 6)(k - 1)$ has a base point free $g^1_d$.

\begin{theorem}\label{s1.0}
Let $C$ be a smooth curve of genus $g\ge 64$ with gonality $k\ge 7$. Then $C$ has a primitive $g^1_{g-4}$, $W^1_{g-4}({C})$ has pure dimension $g-10$ and the general element of every irreducible component of $W^1_{g-4}({C})$ is primitive.
\end{theorem}

\begin{proof}
Since $2(g-4) -g -2 \ge 0$, Brill-Noether theory gives $W^1_{g-4}({C}) \ne \emptyset$ and that any irreducible component $T$ of $W^1_{g-4}({C})$ has dimension at least $2(g-4)-g-2 =g-10$.  Fix a general $R\in T$. As in the case of any irreducible component of any $W^r_d({C})$ we have $\dim |R| =1$. To prove Theorem \ref{s1.0}
it is sufficient to prove that $R$ is base point free, that $h^0(R(q)) = 2$ for every $q\in C$, and that $\dim T =g-10$.

\quad (a) In this step we prove that $R$ has no base points. If $R$ has a base point, then the case $x=4$ of part (b) of Lemma \ref{s1}
gives $\dim (W^1_{10}({C})) \ge 4$.

\quad (a1) Assume $k=7$. Let $h: C\to \mathbb {P}^1$ be any degree $7$ morphism. Since $7$ is a prime number and $g>7\cdot 7 -7-7+1$, the genus formula for integral
curves contained in $\mathbb {P}^1\times \mathbb {P}^1$ shows that $h$ is unique. Let $m$ be the first integer $>7$ such that $C$ has a base point free $g^1_m$.
Since  $\dim (W^1_{10}({C}))> 10-7$ and $\dim (W^1_7({C})) =0$, we have $m\le 10$. Every integral curve of $\mathbb {P}^1\times \mathbb {P}^1$ with bidegree $(7,m)$
has arithmetic genus $7m-7-m+1 \le 70-7-m+1 <g$, a contradiction.

\quad (a2) Assume $k=8$. First assume that $C$ has only finitely many $g^1_8$. Let $h: C\to \mathbb {P}^1$ be any degree $8$ morphism. Since $\dim (W^1_{9}({C})) \ge 2$, we get the existence
of a degree $9$ morphism $f: C \to \mathbb {P}^1$. Since $8$ and $9$ are coprime,  the map $(h,f): C\to \mathbb {P}^1\times \mathbb {P}^1$ is birational onto its image, $\Gamma$. Since
$p_a(\Gamma ) =8\cdot 9-8-9+1$, we get $g\le 56$, a contradiction. Now assume $\dim (W^1_8({C})) \ge 1$ and hence $\dim (W^1_8({C}))=1$. Since $g > 8\cdot 9-8-9+1$, the proof of step (a1) shows that every element of $W^1_9({C})$ has one base point. Since $\dim (W^1_{10}({C})) > 3 + \dim (W^1_8({C}))$, there
is a degree $10$ morphism $\ell : C\to \mathbb {P}^1$.
Since $g > 8\cdot 10 -8-10+1$, we get that $(h,\ell ): C\to \mathbb {P}^1\times \mathbb {P}^1$ has degree $2$ onto its image $\Phi _{(h,\ell )}$ and $\Phi _{(h,\ell )}$ is an integral curve
of bidegree $(4,5)$. Since $k>4$, the normalization $D_{(h,\ell )}$ of $\Phi _{(h,\ell )}$ has genus $>1$. Hence varying $u$ and $\ell$ we only get finitely many smooth curves
$D_{(h,\ell )}$.
We fix one such normalization $D = D_{(h,\ell )}$ and call $v: C\to D$ the map induced by $(u,\ell )$ and $q$ the genus of $D$. We have $q\le 4\cdot 5-4-5+1 =12$ by the Castelnuovo - Severi inequality.
Since $g>2q+9$ all base point free pencils of degree $8$ (resp. $10$)
pencils on $C$ are induced by a degree $4$ (resp. $5$) pencil on $D$ (again by the Castelnuovo - Severi inequality). Since $\dim (W^1_{10}({C})) \ge 4$, we get
$\dim (W^1_5({D})) \ge 4$. Hence $W^1_3({C}) \ne \emptyset$ (\cite{fhl}). Therefore $k\le 6$, a contradiction.

\quad (b) Assume $\dim (T) >g-10$. The case $x=4$ of part (a) of Lemma \ref{s1} gives $\dim (W^1_8({C})) \ge 3$. Hence $k\le 7$ and if $k=7$, then $\dim (W^1_7({C}))=1$.
 Assume $k=7$. Since $7$ is a prime number and $C$ has at least two $g^1_7$, the Castelnuovo - Severi inequality gives $g\le 7\cdot 7-7-7+1$, a contradiction.

\quad ({c}) Assume the existence of $q_R\in C$ such that  $h^0(R(q_R)) =3$. Since $R$ is general in $T$, we get $\dim (W^2_{g-3}({C})) \ge g-11$.
Since $R$ is base point free and $h^0({R}) =2$, $M:= R(q_R)$ is base point free. Let $u: C\to \mathbb {P}^2$ be the morphism induced by $|M|$. Since
$g < (g-4)(g-5)/2$ either $\deg (u) >1$ or $u({C})$ is a singular curve. Therefore taking a linear projection from a point of $u({C})$ (case $\deg (u)>1$)
or a singular point of $u({C})$ (case $\deg (u) =1$), we obtain $\dim (W^1_{g-5}({C})) \ge g-11$. The case $x=5$ of part (a) of Lemma \ref{s1} gives $\dim (W^1_{10}({C})) \ge 4$.
We are in the case excluded in step (a).
\end{proof}

\begin{remark}
Let $C$ be a smooth curve of genus $g$ with a primitive $R\in \mathrm{Pic}^d({C})$, $d\le g-2$. Since $\omega _C\otimes R^\vee$ is primitive, $C$ has a primitive $g^{g-d}_{2g-2-d}$.
In the case $k\ge 5$ and $d=g-2$ for a general $R$ the dual linear series $\omega _C\otimes R^\vee$ is birational onto it image and with image a plane nodal curve
(\cite[Proposition II.0]{v}). See \cite{m1}, \cite{mp}, \cite{ckm2} for
the very ampleness of some $g^3_{g+1}$ (case $d=g-3$).
\end{remark}

\providecommand{\bysame}{\leavevmode\hbox to3em{\hrulefill}\thinspace}

\end{document}